\documentclass[12pt]{amsart}
\usepackage{graphicx}
\usepackage{fullpage}
\usepackage{hyperref}
\usepackage{amsfonts}
\usepackage{amsthm}
\usepackage{amsmath}
\usepackage{mathtools}
\usepackage{float}
\usepackage{overpic}
\usepackage{xcolor}

\newtheorem{theorem}{Theorem}
\newtheorem{claim}{Claim}

\newtheorem{question}{Question}

\theoremstyle{definition}
\newtheorem{definition}{Definition}
\newtheorem{remark}{Remark}

\title{Graphons, Geometry, and Dynamics:\\ Forward and Inverse Perspectives}

\author[\'A. Backhausz]{\'Agnes Backhausz$^{1, 2}$}
\author[C. Kuehn]{Christian Kuehn$^{3, 4, 5}$}
\author[S. van der Niet]{Sjoerd van der Niet$^{1, 2, \ast}$}
\thanks{$^\ast$\href{mailto:van.der.niet.sjoerd@renyi.hu}{van.der.niet.sjoerd@renyi.hu}}

\thanks{The authors acknowledge funding by the European Union’s Horizon Europe Marie Sk\l{}odowska-Curie Actions under the ``BeyondTheEdge: Higher-Order Networks and Dynamics'' project (Grant Agreement No. 101120085).}

\address{$^1$ELTE E\"otv\"os Lor\'and University, Faculty of Science, Institute of Mathematics, Budapest, Hungary}

\address{$^2$HUN-REN Alfr\'ed R\'enyi Institute of Mathematics, Budapest, Hungary}

\address{$^3$Munich Data Science Institute, Garching b.~M\"unchen, Germany}

\address{$^4$Technical University of Munich, School of Computation Information and Technology, Department of Mathematics, Garching b.~M\"unchen, Germany}

\address{$^5$Complexity Science Hub Vienna, Vienna, Austria}

\date{\today}

\begin{document}

\begin{abstract}
    In this work, we explore the interplay between graph limit theory, the geometry of underlying probability spaces, spectral theory, and network dynamical systems. We investigate two primary questions concerning forward and inverse perspectives: first, whether a graphon retains information about the geometry of the space on which it is defined, and second, whether spectral properties can distinguish graphons that originate from different geometric spaces. To address these questions, we differentiate between combinatorial equivalence and geometric structure, highlighting how these concepts are captured simultaneously by the class of pure graphons. Furthermore, we construct explicit examples of isospectral graphons---graphons whose integral operators share the same spectrum---that differ in their underlying geometry. By utilizing the heat kernels of Neumann- and Dirichlet-isospectral drums, we demonstrate that these graphons are not combinatorially equivalent. Finally, we establish new connections between the geometric aspects of graph limit theory and dynamical systems by analyzing a continuum Kuramoto model with graphon-defined interactions. We demonstrate that while isospectrality implies identical stability properties in certain cases, this correspondence breaks down when the differing boundary conditions of our specific Neumann and Dirichlet constructions are considered.
\end{abstract}

\maketitle

\section{Introduction}
Dynamical systems posed on networks have emerged as a crucial building block of a wide variety of models in applied mathematics and beyond~\cite{Newman,Battistonetal1,Berneretal}. Yet, it is often extremely difficult to tackle finite-dimensional systems with a large number of agents or particles, which represent the nodes of the network. This challenge has led to recent substantial progress in the derivation of mean-field~\cite{ChibaMedvedev,KaliuzhnyiVerbovetskyiMedvedev1,GkogkasKuehnXu1,KuehnXu1,AyiDuteilPoyato,KuehnPulido} and continuum~\cite{Medvedev4,GkogkasKuehnXu,BickSclosa,Throm} limits for network dynamical systems, where recent results show that one can aim to track the graph structure also in the infinite-node limit in many cases. This approach forms a bridge between dynamical systems and graph limit theory~\cite{lovasz2012}. In particular, all graph limit concepts (graphons~\cite{Medvedev3}, graphings~\cite{LackerRamananWu1}, graphops~\cite{KuehnGraphops}, graphexes~\cite{FabianCuiKoeppl}, etc) have been studied in some variants of mean-field limit integro-partial-differential equations that arise from network dynamics. In this work, we focus on graphons but our strategy also applies in principle to other classes of graph limits.
A fundamental feature of the theory is that graphons are defined up to measure-preserving transformations, allowing every graphon to be represented on a canonical probability space such as the unit interval [0,1], i.e., the nodes of the graph become points in the unit interval in the graph limit, while the graphon $W:[0,1]\times[0,1]\rightarrow [0,1]$ encodes the edges. Classical measure theory indicates that the approach of mapping any other domain on which a graph limit might be posed to the unit interval is reasonable from a measure-theoretic viewpoint~\cite{janson2010}. In particular, while this representation is sufficient for capturing combinatorial quantities---such as homomorphism densities and convergence in the cut metric~\cite{borgs2008}---there is a natural and largely unresolved question:
\begin{question}\label{question:inverse}
    Does a graphon retain information about the geometry of the space on which it is defined?    
\end{question}
This `forward-influence' question is crucial from the dynamical systems viewpoint as the original network dynamical system usually come with natural geometries. For example, we can think of animals within a swarm. These agents naturally are placed within an ambient geometry, and this geometry influences the  distance between them as well as how they interact via coupling~\cite{CuckerSmale1,VicsekZafiris}. The same argument applies to a wide variety of situations in biology~\cite{Romanczuketal} and it clearly extends to human social interactions~\cite{ChoMyersLeskovec} as well as to engineering systems such as robots and autonomous cars~\cite{Qu}. Discarding geometric information might be acceptable for some parts of dynamical systems analysis but may not be adequate for other aspects. Hence, we must mathematically investigate, which notions of equivalence are available. We have to clarify, when aspects of the geometry can be discarded or when they can even be recovered from the standard graphon representation on the unit square. In this paper, we aim to communicate the answers we found starting from the motivating question above.

A starting point are graphons arising from kernels defined on metric spaces, i.e., we assume that the domain of $W$ is not the unit interval but $W:X\times X\rightarrow [0,1]$ for some metric space $X$. In this setting, the underlying geometry may influence analytic properties of the integral operator associated to $W$. Motivated by existing dynamical systems calculations in a mean-field/continuum limit setting~\cite{ChibaMedvedevMizuhara,KuehnThrom2,Gkogkasetal,BramburgerHolzer,Dueringetal}, one frequently observes that the spectral behavior of the graphon is one aspect that is being used. However, we are going to recall that classical results in graph limit theory show that any graphon defined on a general probability space is weakly isomorphic to one defined on $[0,1]$, preserving all homomorphism densities and, in particular, the spectrum. Yet, this suggests then the inverse problem, i.e., suppose we would know the spectrum of a graphon, how much does it distinguish about the underlying aspects of the geometry that we discarded. A first indication that geometry leaves a nontrivial imprint comes from the study of spectral decompositions. The Spectral Regularity Lemma ensures that any bounded kernel operator admits an approximation by a truncated spectral expansion with controlled error, reinforcing the central role of eigenvalues and eigenfunctions in understanding graphons~\cite{Szegedy2011}. This leads to a second `inverse-influence' guiding question for our work that is also of crucial dynamical importance:

\begin{question}\label{question:forward}
    Can spectral properties distinguish graphons that originate from different geometric spaces?
\end{question}

Indeed, suppose we would have fully analyzed a dynamical mean-field or continuum limit model using just knowledge about the spectrum of a graphon on $[0,1]$, how much can be actually inferred about the underlying geometric structure of the finite-size realistic particle system. This question is particularly important for design and control problems in engineering applications~\cite{Qu} and could also have major impact on biological applications~\cite{Romanczuketal}.

To address our second key question, we distinguish combinatorial equivalence from geometric structure. Weakly isomorphic graphons are indistinguishable with respect to homomorphism densities. However, an earlier work~\cite{lovasz2010} shows that additional structure can be recovered by equipping the underlying space with a metric induced by the graphon itself. In particular, for a broad class of graphons---so-called pure graphons---weak isomorphism implies a much stronger notion of equivalence: isometry of the induced metric spaces. This result demonstrates that, under strong conditions, the geometry of the underlying space is not lost, but rather encoded intrinsically in the graphon.

Although every graphon is weakly isomorphic to a pure graphon~\cite{lovasz2010}, pure graphons do not cover the space of all graphons. Therefore, we should precisely identify cases, where spectral information is insufficient to re-construct the geometry. This observation naturally connects our setting to a classical problem in spectral geometry, famously phrased by Mark Kac as “Can one hear the shape of a drum?” \cite{kac1966}. In our context, this translates to the question of whether the spectrum of a graphon determines its structure up to weak isomorphism, which is a precise formulation of Question~\ref{question:forward}. We investigate this question by constructing explicit examples of isospectral graphons, i.e., graphons whose associated integral operators share the same spectrum, but at the same time they differ in their underlying geometry. Our analysis proceeds in two stages. First, we show that finite-rank constructions allow for straightforward examples of isospectral graphons defined on distinct spaces, such as spheres of different dimensions. We then turn to a more intricate infinite-dimensional setting inspired by the celebrated construction of isospectral but non-isometric planar domains by Gordon, Webb, and Wolpert. By considering graphons induced by heat kernels on these domains, we obtain examples that are isospectral yet not weakly isomorphic, thereby demonstrating that spectral data alone does not determine the graphon.

This last step then naturally brings us back to our first motivating forward-influence question. Now that we have constructed examples of graphons that are isospectral but not weakly isomorphic, we can check in which parts of a dynamical systems analysis this aspect may influence the outcome. We explore the differences by studying a continuum Kuramoto model with graphon-defined interactions. While isospectrality implies identical stability properties in certain cases, we show that this correspondence breaks down when additional structural features such as boundary conditions and the shape of steady states as well as the full linearization process are taken into account. Indeed, this is quite natural since the graphon operator does not appear alone in the stability analysis, other parts of the equation contribute in a (nonlinear and nontrivial) way so that geometry may not always be discarded. This highlights a subtle interplay between spectral, geometric, and dynamical properties of graphons.

The structure of the paper is as follows. In Section~\ref{section:basic theory} we recall concepts of graph limit theory which will be used in the rest of the paper. Section~\ref{section:inverse influence} addresses Question~\ref{question:inverse} by combining results from earlier works relating graph limit theory to geometry, highlighting notions and statements to identify graphons and their intrinsic geometry. In Section~\ref{section:isospectral graphons} we apply these notions to construct various examples of isospectral graphons---including a pair derived from the heat kernels of two isospectral domains---to address Question~\ref{question:forward}. In Section~\ref{section:isospectral dynamics} we analyze dynamical systems modeled by these specific isospectral graphons. Using classical results, we demonstrate that these kernels generate stability operators that are intrinsically non-isospectral.

\section{Basic Graphon and Measure Theory}\label{section:basic theory}
\subsection{Weakly isomorphic graphons}
The space of graphons are famously introduced by completing the space of graphs~\cite{lovasz2006}. Without specifying the notion of convergence for now, take any converging sequence $(G_n)_n$ in the space of graphs, and complete the space by identifying the limit of the sequence. As a result, the limit can be represented in various forms~\cite[Theorem 11.52]{lovasz2012}. To stay close to the intuitive notion of an adjacency matrix, we first consider a graph $G$ on $n$ vertices, whose adjacency matrix is $A\coloneq A(G)$. We can then represent $G$ as a function $W_G:[0,1]^2\to[0,1]$, defined by
\begin{equation*}
    W_G(x,y)=A(\lceil nx\rceil, \lceil ny\rceil).
\end{equation*}
This suggests that our limit object should be defined on the unit interval as well. However, before showing whether this is justified or not, we consider a \textit{graphon} to be a measurable symmetric function $W:\Omega\times\Omega\to[0,1]$, where $(\Omega,\mathcal A, \mu)$ can be any probability space. 

The motivation for identifying ``similar'' graphs---such as Paley graphs and random graphs with edge density $1/2$---comes from the property that the two graphs, asymptotically, contain exactly the same number of copies of each fixed graph $F$~\cite{borgs2008}. This led to the idea of identifying graphs whose densities of these copies coincide. For a graphon $W:\Omega\times\Omega\to[0,1]$ we define the \textit{homomorphism density} of a simple finite graph $F$ in $W$ by
\begin{equation*}\label{eq:homomorphism density def}
    t(F,W) = \int_{\Omega^{V(F)}} \prod_{ij\in E(F)}W(x_i,x_j)\prod_{i\in V(F)}\,\mathrm d\mu(x_i),
\end{equation*}
where we indeed have $t(F,W_G)=t(F,G)\coloneq n^{-|V(G)|}\textup{hom}(F,G)$. This refers back to the start of the section where we speak about the completion of space of graphs. We say that a sequence of graphons $(W_n)_n$ converges---this includes a sequence of graphs $(G_n)_n$ represented by $(W_{G_n})_n$---if the homomorphism densities $(t(F,W_n))_n$ converge for any simple graph $F$. Uniqueness of the limit object highlights the question on how to represent any graphon from a measure-theoretic viewpoint: if $(\Omega,\mathcal A, \mu)$ is \textit{isomorphic} to another probability space $(\Omega',\mathcal A', \mu')$, i.e., there exists a map $\phi:\Omega\to\Omega'$ which is a bijection between $\mathcal A$ and $\mathcal A'$ which preserves the measure, then the object $W^\phi:\Omega'\times\Omega'\to[0,1]$ defined by $W^\phi(\phi(x),\phi(y))$, has exactly the same homomorphism densities as $W$. We note that this limit object is indeed far from unique. This gives rise to the following notion.
\begin{definition}\label{def:weak-isom}
    Two graphons $W_1:\Omega_1\times\Omega_1\to[0,1]$ and $W_2:\Omega_2\times\Omega_2\to[0,1]$ are called \textit{weakly isomorphic} if for every simple graph $F$, we have $t(F,W_1)=t(F,W_2)$.
\end{definition}
The above definition raises the question whether the two graphons $W_1$ and $W_2$ are weakly isomorphic if and only if there exists an isomorphism $\phi:\Omega_1\to\Omega_2$ between the two probability spaces such that $W_1(x,y)=W_2(\phi(x),\phi(y))$ for almost every $x,y\in\Omega_1$. This is generally not true, as shown by~\cite[Example~2]{borgs2010}. As a matter of fact, the example shows that one cannot even find a measure-preserving map, let alone a bijection.

\subsection{The role of the probability space for graphons}
We now aim to answer the question whether it is justified to consider graphons on $([0,1],\lambda)$, the unit interval with the Lebesgue measure, instead of a space that appears more natural. To formalize this, we must recall the classes of some well-behaved spaces. A measurable space is called a \textit{Borel space} if it is isomorphic to a Borel subset of a Polish space (a complete separable metric space). A \textit{Borel probability space} is simply a Borel space equipped with a probability measure. A \textit{standard probability space} is defined as the completion of a Borel probability space.
The argument that the homomorphism densities are invariant under an isomorphic transformation of the measure space, extends to the notion of probability spaces $(\Omega,\mathcal A, \mu)$ and $(\Omega',\mathcal A', \mu')$ which are \textit{isomorphic up to nullsets}, i.e., we can delete sets of measure $0$ from $\mathcal A$ and $\mathcal A'$ such that the remaining probability spaces are isomorphic. It is a well-known fact that any standard probability space is isomorphic up to nullsets to the disjoint union of the $[0,1]$ interval with the Lebesgue measure and countably many atoms. As a result, any graphon on a standard probability space without atoms is weakly isomorphic to a graphon defined on $([0,1],\lambda)$. For a more detailed discussion, see Appendix~A in~\cite{janson2010}.

\begin{remark}\label{rem:spectrum identical}
    It is important to highlight here, that such an embedding into $([0,1],\lambda)$ does not only leave the homomorphism densities invariant, but various quantities as well, in particular the spectrum associated with a graphon. Indeed, the integral operator $T_W:L^2(\Omega)\to L^2(\Omega)$ defined by
    \begin{equation*}
        T_Wf(x)=\int_\Omega W(x,y)f(y)\,\mathrm d\mu(y),
    \end{equation*}
    has the same spectrum as $T_{W^\phi}:L^2([0,1]
    )\to L^2([0,1])$, and the eigenfunctions are one-to-one related through the pushforward $U_\phi$ of $\phi$, the isomorphism up to nullsets. For example, if $\Omega$ is a Riemannian manifold with finite volume and $\mu$ is the Riemannian volume measure, where $W$ respects the Riemannian metric $g$, then its eigenfunctions will behave nicely with respect to $g$ as well. However, applying $U_\phi$ to any eigenfunction $f$ of $W$, associated with a nonzero eigenvalue, might yield a fractal-like eigenfunction $U_\phi f$ of the transformed graphon $W^\phi$, which itself may also appear fractal-like. This is due to the guaranteed nowhere continuous maps from higher-dimensional spaces to $[0,1]$, erasing the ``visual'' trace of $g$. We aim to study these observations in the following sections.
\end{remark}

Although it might be convenient that such a map exists and ensures that many properties are preserved, including the identical spectrum of the two graphons, in~\cite{janson2010} they prove a somewhat surprising but stronger fact from a combinatorial point of view.
\begin{theorem}[{\cite[Theorem 7.1]{janson2010}}]
    Every graphon on a probability space $(\Omega,\mathcal A,\mu)$ is weakly isomorphic to a graphon on $([0,1],\lambda)$.
\end{theorem}

Note that $(\Omega,\mathcal A,\mu)$ is allowed to be a completely generic probability space. Although the existence of the measure-preserving bijection between the two probability spaces is not guaranteed here. However, using the techniques to prove the above theorem, it is not hard to conclude that the spectrum of the two graphons are still identical up to the multiplicity of the eigenvalue $0$.

\subsection{Graphon metrics and quotient space}
With the arguments above, one could argue it that is enough to consider only graphons on $([0,1],\lambda)$ from a combinatorial or measure-theoretic point of view.
This school of thought provides a convenient way to metrize graphons based around the following norm
\begin{equation*}\label{eq:cut norm def}
    \|W\|_\square=\sup_{\|f\|_\infty,\|g\|_\infty\leq 1}\bigg|\int_{[0,1]^2} W(x,y)f(x)g(y)\,\mathrm d x\,\mathrm d y\bigg|,
\end{equation*}
which is called the \textit{cut norm}. Note that we define it for $([0,1],\lambda)$, but all results remain true when considering a general probability space.

Reshuffling the nodes of a graph does not affect the homomorphism densities, so neither should it for a graphon. Any measure-preserving map $\varphi:[0,1]\to[0,1]$ leaves the homomorphism densities invariant, however, it results usually in $\|W-W^\varphi\|_\square\neq 0$, where we define $W^\varphi:[0,1]^2\to[0,1]$ by $W^\varphi(x,y)=W(\varphi(x),\varphi(y))$.
Hence we define the \textit{cut metric} as
\begin{equation*}
    \delta_\square(U,W) = \inf_{\varphi}  \|U-W^\varphi\|_\square,
\end{equation*}
where the infimum is taken over all measure-preserving maps $\varphi:[0,1]\to[0,1]$.
\begin{theorem}[{\cite[Theorem 3.8]{borgs2008}}]\label{thm:conv-equiv}
    Convergence of homomorphism densities is equivalent to convergence in the cut metric.
\end{theorem}
So two graphons are weakly isomorphic if and only if their cut distance is zero.
If $\mathcal W$ is the space of all symmetric measurable functions $W:[0,1]^2\to[0,1]$ (graphons), then $\widetilde{\mathcal{W}}\coloneqq\mathcal W / \sim$ is the (quotient) space of graphons where $U\sim W$ if $U$ and $W$ are weakly isomorphic.
Moreover, $(\widetilde{\mathcal{W}},\delta_\square)$ is compact and the set of graphs is indeed dense in $(\widetilde{\mathcal{W}},\delta_\square)$~\cite{lovasz2012}.

We conclude this section with the following. We have a graphon $W$ defined on a general probability space. We can consider a weakly isomorphic representation $\widetilde W$ on $[0,1]$, which we can approximate by a sequence of finite graphs $(G_n)$, in the sense that $t(F,G_n)\to t(F,\widetilde W)=t(F,W)$ for any $F$. So they represent the same limit object in a combinatorial sense. If one restricts slightly further to an atomless standard probability space, the resulting $W^\phi$ is also identical in a measure-theoretic framework. So from both a combinatorial and measure-theoretic point of view, any such representation seems reasonable, as hinted in Remark~\ref{rem:spectrum identical}. In the following sections, we identify specific situations where this choice requires extra attention.

%\section{Spectral Observations and inverse-Influence}\label{section:spectral observations}
%[REVISE THESE COMMENTS LATER: Is this section really necessary?]

%These questions have partly been addressed in~\cite{Szegedy2011}. It build on the following adaptation of the famous Regularity Lemma by Szemer\'edi.
%\begin{theorem}[Spectral Regularity Lemma]
%     For an arbitrarily decreasing function $F:\mathbb R^+\times\mathbb R^+\to\mathbb R^+$ and every $\varepsilon>0$ there is a constant $\delta > 0$ such that for every self adjoint kernel operator $M:\Omega\times\Omega\to\mathbb C$ with $\|M\|_\infty\leq 1$ on a separable probability space $(\Omega,\mu)$ there is a real number $\lambda\geq\delta$ such that $M$ has a decomposition $M= S+E+R$ with the following properties
%     \begin{enumerate}
%         \item $S=\sum_{i:|\lambda_i|\geq \lambda}\lambda_i f_if_i^*$
%         \item $\|E\|_2\leq\varepsilon$
%         \item $\|R\|_\square\leq F(\lambda,\varepsilon)$
%         \item $\|S+E\|_\infty\leq 1$
%         \item $E$ and $R$ are self adjoint,
%     \end{enumerate}
%     where the $\{f_i\}$ are the eigenfunctions of $M$ associated with eigenvalues $\{\lambda_i\}$ and form an orthonormal basis in $L^2$.
% \end{theorem}
% It roughly says: we can approximate any graphon $W$ with its truncated spectral expansion $S$, while ensuring the $L^2$-error $E$ and the combinatorial error $R$ are kept arbitrarily small. The proof of this theorem already highlights one of the hurdles we wished to identify. They assume without loss of generality that $\Omega$ is some standard probability space.

\section{Geometric Structure vs. Combinatorial Equi\-valence and inverse-Influence}\label{section:inverse influence}
In this section we wish to address Question~\ref{question:inverse}. We do this by recalling various notions from~\cite{lovasz2010} on how to obtain  graphons, so called ``pure graphons'', which are weakly isomorphic and such that their underlying probability spaces are isomorphic, identifying the intrinsic geometry in some sense. But first we provide some background on where the underlying geometry might leave a trace in the spectral expansion of a graphon.

Question~\ref{question:inverse} was already studied in~\cite{Szegedy2011} by means of spectral decomposition. In particular, their Spectral Regularity Lemma says that any bounded self-adjoint kernel operator $M:V\times V\to\mathbb C$ (this includes graphons) can be decomposed as $M= S+E+R$, where $S$ is a finite-rank truncated spectral expansion of $M$. The kernels $E$ and $R$ are remainder terms, where $E$ controls the $L^2$-error made by the spectral expansion and $R$ the ``combinatorial error'' in the cut norm. This provides a tool to study graphons, which can be analyzed easily via their dominant eigenspaces, a quantity which is known to carry information on the underlying space as shown below.

The following example arises from a clever application of the Spectral Regularity Lemma to graphons which are invariant under a group of unitary operations. The representation theory of such group actions allows the authors to analyze certain sets of graphons in a specific way, which comes down to analyzing the dimensions of the eigenspaces of the graphon induced by the group action. To be specific, they define $\mathcal S_n^0$ to be the space of graphons on the sphere $\mathbb S^n=\{x\in\mathbb R^{n+1}: \|x\|_2=1\}$ with the uniform measure, which is invariant under the induced action of the orthogonal group $O_{n+1}$, i.e., $W\in \mathcal S_n^0$ if $W(x,y)=f(\langle x,y\rangle)$ for some function $f:[-1,1]\to[0,1]$.

Using the representation theory of $O_{n+1}$ (the dimensions of the spherical harmonics to be precise), they show that $\mathcal S_n^0$ is compact for $n\geq 2$, whereas $\mathcal S_1^0$ is not. 
As a matter of fact, $\mathcal S_1^0$ is not even closed with respect to the cut metric.
It is possible to construct a sequence of graphons in $\mathcal S_1^0$ whose underlying topology is not $\mathbb S^1$.

This leads us to the question what the underlying topology of a graphon is. To emphasize that graphons are defined on different spaces, we denote a graphon $W:\Omega\times\Omega\to[0,1]$ by $(\Omega,W)$. 
Consider a graphon $(\mathbb S^2,W)\in \mathcal S_2^0$.
We have established that there exists a weakly isomorphic graphon $(\mathbb S^1,W^\varphi)$ via some measure-preserving map $\varphi:\mathbb S^1\to \mathbb S^2$ (which can even be a Borel isomorphism). This raises the question whether it is even necessary to consider the underlying topology of a graphon.

Again consider two graphons $(\mathbb S^1,W_1)\in\mathcal S_1^0$ and $(\mathbb S^2,W_2)\in\mathcal S_2^0$.
In~\cite{lovasz2010} they introduce the following metric $r_W$ on the underlying space $\Omega$ of a graphon $(\Omega, W)$ as
\begin{equation*}
    r_W(x,y) = \int_\Omega |W(x,z)-W(y,z)|\,\mathrm d\mu(z).
\end{equation*}
They note that, specifically for $(\mathbb S^1,W_1)$ and $(\mathbb S^2,W_2)$, the identity map $\iota: (\mathbb S^i, r_{W_i})\to(\mathbb S^i,r_i)$ is a homeomorphism for $i=1,2$, where $r_i$ is the geodesic distance on $\mathbb S^i$. Obviously, the circle and the sphere are not homeomorphic, and in~\cite{lovasz2010} they establish a strong implication for weakly isomorphic graphons satisfying a specific set of conditions, which helps to identify the underlying topology.
\begin{definition}
    A graphon $(\Omega,W)$ is called \textit{pure} if $(\Omega,r_{W})$ is a complete separable metric space and the probability measure $(\Omega,\mu)$ has full support (i.e., every open set has positive measure).
\end{definition}
A meaningful observation is that this definition forbids so called \textit{twin points}, i.e., there are no $x,y\in\Omega$ such that $W(x,\cdot)$ and $W(y,\cdot)$ are equal almost everywhere. A graphon with no twin points is called \textit{twin-free}. We also need the following definition.
\begin{definition}
    Two graphons $(\Omega,W)$ and $(\Omega',W')$ are called \textit{isometric} if there exists an isometric bijection $\phi: (\Omega,r_W)\to(\Omega',r_{W'})$ that is measure-preserving, and $W'(\phi(x), \phi(y)) = W (x, y)$ for almost all $x,y\in \Omega$.
\end{definition}
If such a map exists, then the two graphons are obviously weakly isomorphic, as it implies $\delta_\square(W,W')=0$ (Definition~\ref{def:weak-isom} and Theorem~\ref{thm:conv-equiv}). We get the following statement from~\cite{lovasz2010}.
\begin{theorem}\label{thm:isomorphic graphons}
    If two pure graphons are weakly isomorphic, then they are isometric.
\end{theorem}
We interpret this result by coming back to the case of $(\mathbb S^1,W_1)\in\mathcal S_1^0$ and $(\mathbb S^2,W_2)\in\mathcal S_2^0$. Since the two metric spaces $(\mathbb S^1, r_{W_1})$ and $(\mathbb S^2, r_{W_2})$ are homeomorphic to the circle and the sphere with geodesic distances respectively, and these two objects are not homeomorphic, there exists no isometry between $(\mathbb S^1, r_{W_1})$ and $(\mathbb S^2, r_{W_2})$. As a result, $(\mathbb S^1,W_1)$ and $(\mathbb S^2,W_2)$ can never represent the same object (i.e., be weakly isomorphic), due to their constructions intrinsically depending on the geometry of the underlying spheres.

This exposes the subtle interplay between geometric structure and combinatorial equivalence. While the intrinsic geometry of the underlying spaces can prevent two graphons from being weakly isomorphic (combinatorially equivalent), this does not forbid a representation on a canonical domain such as $[0,1]$. Indeed, no matter how the graphon $W:\Omega\times\Omega\to[0,1]$ depends on the geometry of the space $(\Omega,d)$, we can always represent it by a graphon $W^\phi:[0,1]^2\to[0,1]$ which is weakly isomorphic to $W$. Moreover, $W^\phi$ encodes the same properties, including, but not limited to, the spectrum, the degree function, the $L^p$ and cut norms, and all subgraph densities.

As a matter of fact, even the metric spaces $(\Omega,r_W)$ and $([0,1],r_{W^\phi})$ are isometric. No matter how one flattens the graphon, it will always encode the underlying geometry. However, it is not guaranteed that $W$ admits a particularly nice representation in terms of $(\Omega,d)$.

\section{Isospectral graphons}\label{section:isospectral graphons}
This interplay between geometry and the structural and analytical properties of a graphon suggests that we can link the famous question phrased by Mark Kac: ``Can one hear the shape of a drum?''. This is exactly Question~\ref{question:forward}, phrased in a less poetic way as, can two non-weakly isomorphic graphons $W_1$ and $W_2$ be \textit{isospectral}, i.e., do the integral operators $T_{W_i}(f)=\int_{\Omega_i} W_i(\cdot,y)f(y)\,\mathrm d\mu(y)$ have the same spectral measure for $i=1,2$? We first demonstrate how this question is easily answered for \textit{finite-rank} graphons (having only a finite number of nonzero eigenvalues), and how the solution to Mark Kac's problem provides an answer when considering the \textit{infinite-rank} case.

We can construct isospectral graphons by conveniently selecting a superposition of functions which decompose easily in for example the harmonics of a sphere $\mathbb S^n$. We first consider $(\mathbb S^1, W_{\mathbb S^1})$ on the circle, for which we know that $L^2(\mathbb S^1)$ (with the uniform probability measure) has the usual basis given by $\cos(n\cdot)$ and $\sin(n\cdot)$ for $n\geq 0$. As a result, the graphon
\begin{equation*}
    W_{\mathbb S^1}(x,y)= \frac{1}{2}+\frac{1}{8}\cos(d_{\mathbb S^1}(x,y)) + \frac{1}{8}\cos(2d_{\mathbb S^1}(x,y)),
\end{equation*}
has the nonzero spectrum $\{1/2,1/16\}$, where $1/2$ has multiplicity 1 and $1/16$ has multiplicity 4. Note that $\cos(d_{\mathbb S^1}(x,y))=\langle x,y\rangle$ for points on the unit circle in $\mathbb R^2$.

Similarly, we can construct $(\mathbb S^3, W_{\mathbb S^3})$ on the sphere in $\mathbb R^4$ by
\begin{equation*}
    W_{\mathbb S^3}(x,y)= \frac{1}{2}+\frac{1}{4}\cos(d_{\mathbb S^3}(x,y)),
\end{equation*}
where $d_{\mathbb S^3}$ is the geodesic distance on $\mathbb S^3$. Just as with the circle, this means that $W_{\mathbb S^3}(x,y)=\frac{1}{2}+\frac{1}{4}\langle x,y\rangle$. With this expression it easily follows that $W_{\mathbb S^3}$ has the same spectrum as $W_{\mathbb S^1}$. Since both graphons are of the form $W(x,y)=F(d(x,y))$ for some continuous $F:[0,\pi]\to[0,1]$, we get that $(\mathbb S^1,r_{W_{\mathbb S^1}})$ is homeomorphic to $(\mathbb S^1,d_{\mathbb S^1})$. Similarly, $(\mathbb S^3,r_{W_{\mathbb S^3}})$ is homeomorphic to $(\mathbb S^3,d_{\mathbb S^3})$. As a result, since these are pure graphons, they cannot be weakly isomorphic due to Theorem~\ref{thm:isomorphic graphons}.

As demonstrated here, it is easy to force two clearly distinct graphons to be isospectral, as long as they are of finite-rank: the range of $T_W$ is a finite dimensional subspace of $L^2(\Omega)$. If we require $T_W$ to be of infinite-rank, the original negative answer to the question posed by Mark Kac provides an interesting counterexample for us as well. In~\cite{GWW}, the authors construct two domains (drums) $\Omega_1$ and $\Omega_2$ that have identical spectra for the Laplacian operator on the drums, under both Dirichlet and Neumann boundary conditions. In the following we show that these drums give rise to isospectral graphons, which are not weakly isomorphic.
Our specific construction utilizes heat kernels on the drums, allowing us to apply the established literature to the graphons. In turn, this provides additional insight into the spectral properties inherited from the geometry of the domains.

\subsection{The heat kernel on isospectral drums}
Let $\Omega_1,\Omega_2\subset \mathbb R^2$ denote the two (compact) domains constructed by Gordon--Webb--Wolpert in~\cite{GWW}, and equip them with the uniform probability measures $\mu_1$ and $\mu_2$ (induced by the normalized Lebesgue measure). The Dirichlet and Neumann Laplace operators on these domains are proven to be isospectral, i.e., they share the same set of eigenvalues $0\leq\lambda_0\leq\lambda_1\leq\ldots$, with $\lambda_i\to\infty$. We shall refer to these domains as \textit{Dirichlet-isospectral drums} and \textit{Neumann-isospectral drums}, respectively, depending on the boundary conditions applied to the Laplacian.

Because the drums are isospectral, the eigenvalues of the integral operators $T_{H_i}$ given by the heat kernels $H_i(\cdot,\cdot,t)=e^{-\Delta_i t}$, coincide and are given by $\{e^{-\lambda_n t}\}_{n\geq 0}$. Hence, for any $t_0> 0$ we can normalize the kernels by the same value such that both take values in $[0,1]$, i.e., we define the graphons $(\Omega_i,W_i)$ by
\begin{equation*}
    W_i(x,y) = \frac{1}{K} H_i(x,y,t_0),
\end{equation*}
where $K=\max_{i,x,y}H_i(x,y,t_0)$, which gives the spectra $\{\frac{1}{K}e^{-\lambda_n t_0}\}_{n\geq 0}$. We allow for $t_0>0$ to be chosen accordingly, e.g., in Section~\ref{section:Dirich isospec dynam}.

We can show that these two graphons are not isomorphic, by showing they are not isometric. By Theorem~\ref{thm:isomorphic graphons}, this would be true for pure graphons. If both kernels are twin-free---meaning $W_i(x,\cdot)$ and $W_i(y,\cdot)$ differ on a set of positive measure for all distinct $x,y$---then this fact is obvious. In order to avoid twin points on the boundary, we consider the Neumann-isospectral drums for now, but with some care, as described in Remark~\ref{rem:dirch weakly-isom}, we can consider the Dirichlet-isospectral drums as well.
\begin{claim}\label{claim:Neu twin free}
    The graphons $(\Omega_i,W_i)$ are twin-free.
\end{claim}
\begin{proof}
    Let $\Omega$ be any compact connected domain in $\mathbb R^2$ and let $H(x,y,t_0)$ be the heat kernel for $x,y\in\Omega$ evaluated at some $t_0>0$. Then we have the following pointwise expansion
    \begin{equation*}
        H(x,y,t_0)=\sum_{n=0}^\infty e^{-\lambda_nt_0}\psi_n(x)\psi_n(y),
    \end{equation*}
    where $\psi_n$ is the Laplace eigenfunction associated with the Laplace eigenvalue $\lambda_n$. We can choose $\{\psi_n\}_{n\geq 0}$ such that it forms an orthonormal basis in $L^2(\Omega)$. Moreover, $H(\cdot,\cdot,t_0)$ is real-analytic on the interior of the domain $\Omega^\circ$~\cite[Vol. 2, Appendix 1 to Chapter V, \S 4.2]{courant1989}.
    
    Suppose that $W(x,y)=\frac{1}{K}H(x,y,t_0)$ is not twin-free, for some $K>0$. Then there is some set of positive measure $A\subset\Omega$ such that $W(x,z)=W(y,z)$ for every $z\in A$, where $x\neq y$.
    Since $z\mapsto H(x,z,t_0)-H(y,z,t_0)$ is analytic on $\Omega^\circ$, we find that $H(x,z,t_0)=H(y,z,t_0)$ for any $z\in\Omega^\circ$ by~\cite{mityagin2015}. By orthogonality of the eigenfunctions, this implies that
    \begin{equation*}
        \psi_k(x) = e^{\lambda_k t_0}\int_\Omega H(x,z,t_0)\psi_k(z)\,\mathrm{d}z = e^{\lambda_k t_0}\int_\Omega H(y,z,t_0)\psi_k(z)\,\mathrm{d}z =\psi_k(y),
    \end{equation*}
    for every $k\geq 0$. Now take a smooth, compactly supported test function $f$ with $f(x)\neq f(y)$. By~\cite[Vol. 1, Chapter VI, \S 2]{courant1989}, the expansion $f=\sum_n c_n\psi_n$ converges uniformly, so that we find the contradiction
    \begin{equation*}
        f(x)=\sum_{n\geq 0}c_n\psi_n(x) = \sum_{n\geq 0}c_n\psi_n(y)=f(y),
    \end{equation*}
    hence the claim follows.
\end{proof}
What is left to show is that the two graphons are not isometric. This fact follows from the short-time asymptotics of the heat kernel, given by the celebrated Varadhan's formula for the heat kernel $H$ on a Riemannian manifold $M$
\begin{equation*}
    \lim_{t\to0^+}-4t\log H(x,y,t)=d(x,y)^2,
\end{equation*}
where $d$ is the geodesic distance on $M$~\cite{varadhan1967, hsu2002}. This formula holds both for a Dirichlet and Neumann boundary condition. This is where the non-isospectral drums come in, as it is exactly this metric for which they are not isometric.
\begin{claim}\label{claim:Neu isometric}
    The graphons $(\Omega_i,W_i)$ are not isometric.
\end{claim}
\begin{proof}
    Suppose for a contradiction that there exists a $t_0>0$ such that the two graphons are isometric, i.e., there exists $\phi:\Omega_1\to\Omega_2$ such that $H_1(x,y,t_0)=H_2(\phi(x),\phi(y),t_0)$ almost everywhere. Because $\phi$ is measure-preserving we have $T_{H_1} = U T_{H_2} U^{-1}$, where $U:L^2(\Omega_1)\to L^2(\Omega_2)$ is the push-forward of $\phi$ defined as $Uf(x)= f(\phi(x))$. Since this is a unitary operator, we obtain that $U\psi^{(1)}_n = \psi_n^{(2)}$ maps the eigenvectors $\{\psi^{(1)}_n\}$ of $T_{H_1}$ to $\{\psi^{(2)}_n\}_n$, the eigenvectors of $T_{H_2}$. Hence $H_1(x,y,t)=H_2(\phi(x),\phi(y),t)$ holds almost everywhere for all $t>0$.
    
    To derive the contradiction, note that $\Omega_1$ and $\Omega_2$ are not isometric in the sense that $d_1(x,y)\neq d_2(\phi(x),\phi(y))$ for some $x\neq y$ in $\Omega_1$. So there must exist $S\subset\Omega_1\times\Omega_2$ of positive measure on which 
    \begin{equation*}
        |d_1(x,y)-d_2(\phi(x),\phi(y))|\geq \delta,
    \end{equation*}
    for some $\delta>0$, as $\phi$ cannot be an isometry almost everywhere which would then extend to a global isometry. Now choose $t$ small enough such that both
    \begin{equation*}
        |-4 t\log H_1(x,y,t) - d_1(x,y)^2|<\delta/3,
    \end{equation*}
    and
    \begin{equation*}
        |-4 t\log H_2(\phi(x),\phi(y),t) - d_2(\phi(x),\phi(y))^2|<\delta/3.
    \end{equation*}
    Since this violates $H_1(x,y,t)=H_2(\phi(x),\phi(y),t)$ on $S$, we arrive at a contradiction.
\end{proof}
From these two claims, we can conclude that the graphons are not weakly isomorphic, i.e., $t(F,W_1)\neq t(F,W_2)$ for some graph $F$. Notably, $F$ cannot be a cycle, as cycle densities are spectrally invariant.
\begin{remark}\label{rem:dirch weakly-isom}
    The computations in this section are all valid for heat kernels with Dirichlet boundary conditions, provided that we slightly modify these domains to avoid twin points. Because a Dirichlet boundary condition forces the heat kernel $H_i(x,y,t)=0$ whenever $x$ or $y$ is in $\partial\Omega_i$, all boundary points are indistinguishable with respect to $r_{W_i}=\|W_i(x,\cdot)-W_i(y,\cdot)\|_1$ and form a set of twin points. To resolve this, we define the modified domain $\widetilde \Omega_i \coloneqq(\Omega_i\setminus\partial\Omega_i)\cup\{B_i\}$, which collapses the entire boundary into a single point $B_i$. Consequently, any sequence in the interior converging to $\partial\Omega_i$ naturally converges to $B_i$ under the modified metric induced by $\widetilde W_i$. Using the Dirichlet heat kernels $H_i$, we then define the twin-free graphons $\widetilde W_i:\widetilde\Omega_i\times \widetilde\Omega_i\to[0,1]$ by
    \begin{equation*}
        \widetilde W_i(x,y) = \begin{cases} 
        \frac{1}{K}H_i(x,y,t_0), & \text{if } x, y \ne B_i, \\ 
        0, & \text{otherwise}, 
        \end{cases}
    \end{equation*}
    where $K=\max_{i,x,y}H_i(x,y,t_0)$.
    
    The proof of Claim~\ref{claim:Neu twin free} remains identical for interior points $x,y \in \Omega_i \setminus \partial\Omega_i$. Furthermore, for the boundary point $x=B_i$ and an interior point $y \in \Omega_i \setminus \partial\Omega_i$, we observe that $H_i(y,z,t_0)>0$ for any interior point $z$ (whereas $H_i(B_i,z,t_0) = 0$), which means $\widetilde W_i$ is indeed entirely twin-free. Claim~\ref{claim:Neu isometric} remains unchanged, as one would expect from the classic principle of `not feeling the boundary'.
\end{remark}

\section{Isospectral Dynamics and Forward-Influence}\label{section:isospectral dynamics}
To further demonstrate how these isospectral graphons interact with the stability properties of a steady-state solution in a dynamical system, we analyze synchronized rotating frames in a Kuramoto toy model. We consider the following system whose interactions are determined by a graphon $W$ defined on the probability space $(\Omega,\mu)$ 
\begin{equation*}
    \partial_t\theta(x,t)=\omega(x)+\int_\Omega W(x,y)\sin(\theta(y,t)-\theta(x,t))\,\mathrm d \mu(y),
\end{equation*}
where $\omega$ is the intrinsic frequency distribution of the oscillators. Assuming there exists a synchronized steady-state solution $\theta^*(x,t)=\bar\omega t+\psi(x)$ for some collective frequency $\bar\omega$ and phase profile $\psi$, we consider the perturbed solution
\begin{equation*}
    \theta(x,t) = \bar\omega t + \psi(x) + \varepsilon\eta(x,t),
\end{equation*}
such that we obtain the linearized equation
\begin{equation}\label{eq:linearized KM}
    \partial_t \eta(x,t) = \int_\Omega W(x,y)\cos(\psi(y)-\psi(x))(\eta(y,t) - \eta(x,t))\,\mathrm{d} \mu(y).
\end{equation}

Now suppose we assume identical frequencies $\omega(x)\equiv\bar\omega\in\mathbb R$ and $\psi(y)\equiv\psi(x)\mod 2\pi$. Then the linearized dynamics~\eqref{eq:linearized KM}  simplify to a diffusion operator
\begin{equation*}
    K_W \eta(x)= \int_{\Omega}W(x,y)(\eta(y)-\eta(x))\,\mathrm d\mu(y) = T_W\eta(x) - D_W\eta(x),
\end{equation*}
where $D_W$ is the multiplication operator defined as $D_Wf(x) = d_W(x)f(x)$ with the degree function given by $d_W(x)=\int_{\Omega}W(x,y)\,\mathrm d\mu(y)$. We use this operator to compare the two cases: Neumann-isospectral drums and Dirichlet-isospectral drums.

\subsection{Neumann-isospectral drums}
In the case of either Gordon--Webb--Wolpert drum with the Neumann boundary condition, we observe the following. The heat kernel expansion
\begin{equation*}
    H(x,y,t)=\sum_{n\geq 0}e^{-\lambda_n t}\psi_n(x)\psi_n(y),
\end{equation*}
converges uniformly for $t_0>0$, so we find
\begin{equation*}
    d_{W}(x)=\sum_{n=0}^\infty e^{-\lambda_n t_0}\psi_n(x)\int_{\Omega}\psi_n(y)\,\mathrm{d}\mu(y) = e^{-\lambda_0 t_0}\psi_0(x)\int_\Omega\psi_0(y)\,\mathrm d \mu(y) =  1,
\end{equation*}
as we choose $\{\psi_n\}_{n\geq0}$ to be an orthonormal basis, which makes $\psi_0$ constant on $\Omega$. As a result, the linearized operator becomes $K_W=T_W-I$, and the spectrum is explicitly determined by $T_W$. Hence, for the Neumann-isospectral drums, the steady-states $\theta_i^*(x,t) = \omega_i+\psi_i(x)$ with $\psi_i(y)\equiv\psi_i(x)\mod 2\pi$ also induce isospectral stability operators.

\subsection{Dirichlet-isospectral drums}\label{section:Dirich isospec dynam}
Because the Dirichlet-isospectral drums force the \linebreak graphons to vanish on the boundary, we do not obtain the same expression as in the Neumann case, and we cannot expect the operators $D_{W_i}$ to be isospectral. As a matter of fact, these multiplication operators have continuous spectra, whose support is exactly determined by the range of $d_{W_i}$. Since the continuous part of $K_{W_i}$'s spectrum is solely determined by the continuous spectrum of $D_{W_i}$, it suffices to show that $d_{W_i}$ attain different maxima, as $d_{W_i}(x)=0$ on $\partial \Omega_i$ and positive on the interior of $\Omega_i$.

We simulated $H_i(\cdot,\cdot,t)=e^{-\Delta_i t}$ for the two drums using a finite element method with a maximal area of $10^{-6}$ and determined the distribution of $d_{W_i}(\Omega_i)$. Although there are observable computation errors using this grid size, Figure~\ref{fig:drums heat profile} suggests that this does not explain the difference in the maximal value attained by the $d_{W_i}$. If the drums attain different maximal values, then the support of the continuous spectrum of $D_{W_i}$ are different, hence that of $K_{W_i}$ as well.
\begin{figure}[h!]
    \centering
    \begin{overpic}[width=0.9\linewidth, percent]{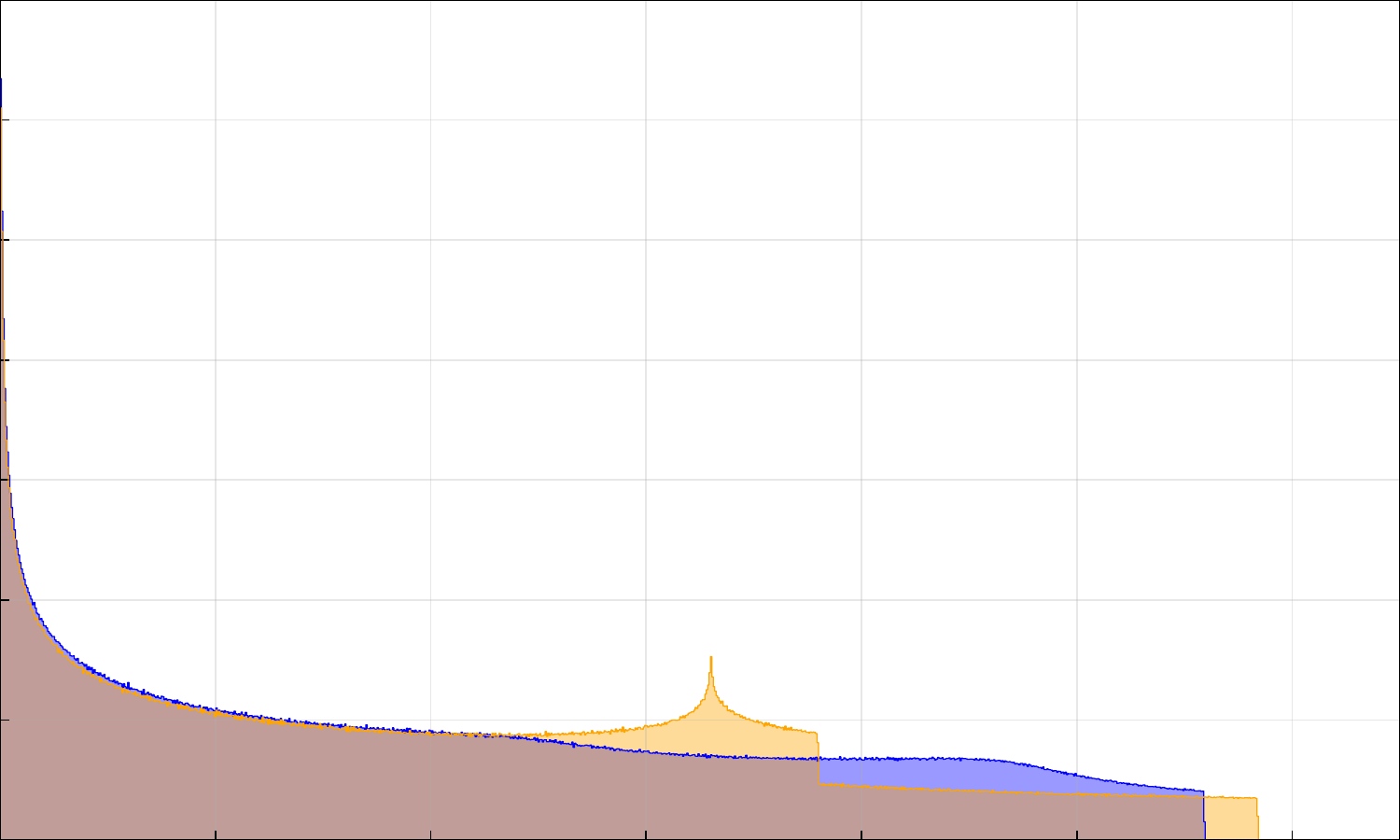}

        % --- X-AXIS TICKS ---
        \put(0.00, -2.5){\makebox(0,0)[t]{$0$}}
        \put(15.38, -2.5){\makebox(0,0)[t]{$0.1$}}
        \put(30.77, -2.5){\makebox(0,0)[t]{$0.2$}}
        \put(46.15, -2.5){\makebox(0,0)[t]{$0.3$}}
        \put(61.54, -2.5){\makebox(0,0)[t]{$0.4$}}
        \put(76.92, -2.5){\makebox(0,0)[t]{$0.5$}}
        \put(92.31, -2.5){\makebox(0,0)[t]{$0.6$}}

        % --- Y-AXIS TICKS ---
        \put(-1.5, 8.57){\makebox(0,0)[r]{\raisebox{-0.7ex}{$2$}}}
        \put(-1.5, 17.14){\makebox(0,0)[r]{\raisebox{-0.7ex}{$4$}}}
        \put(-1.5, 25.71){\makebox(0,0)[r]{\raisebox{-0.7ex}{$6$}}}
        \put(-1.5, 34.29){\makebox(0,0)[r]{\raisebox{-0.7ex}{$8$}}}
        \put(-1.5, 42.86){\makebox(0,0)[r]{\raisebox{-0.7ex}{$10$}}}
        \put(-1.5, 51.43){\makebox(0,0)[r]{\raisebox{-0.7ex}{$12$}}}

        % --- LABELS & LEGEND ---
        \put(50, -10.5){\makebox(0,0)[t]{Value of $d_W(x)$}}
        \put(-9, 30){\makebox(0,0)[b]{\rotatebox{90}{Density}}}
        \put(80, 54){\textcolor{blue!40}{\rule{8pt}{8pt}}\; Drum 1}
        \put(80, 47){\textcolor{orange!40}{\rule{8pt}{8pt}}\; Drum 2}
    \end{overpic}
    \vspace{4em}
    \caption{Distribution of $d_{W_i}(\Omega_i)$ for the two Dirichlet-isospectral drums.}
    \label{fig:drums heat profile}
\end{figure}
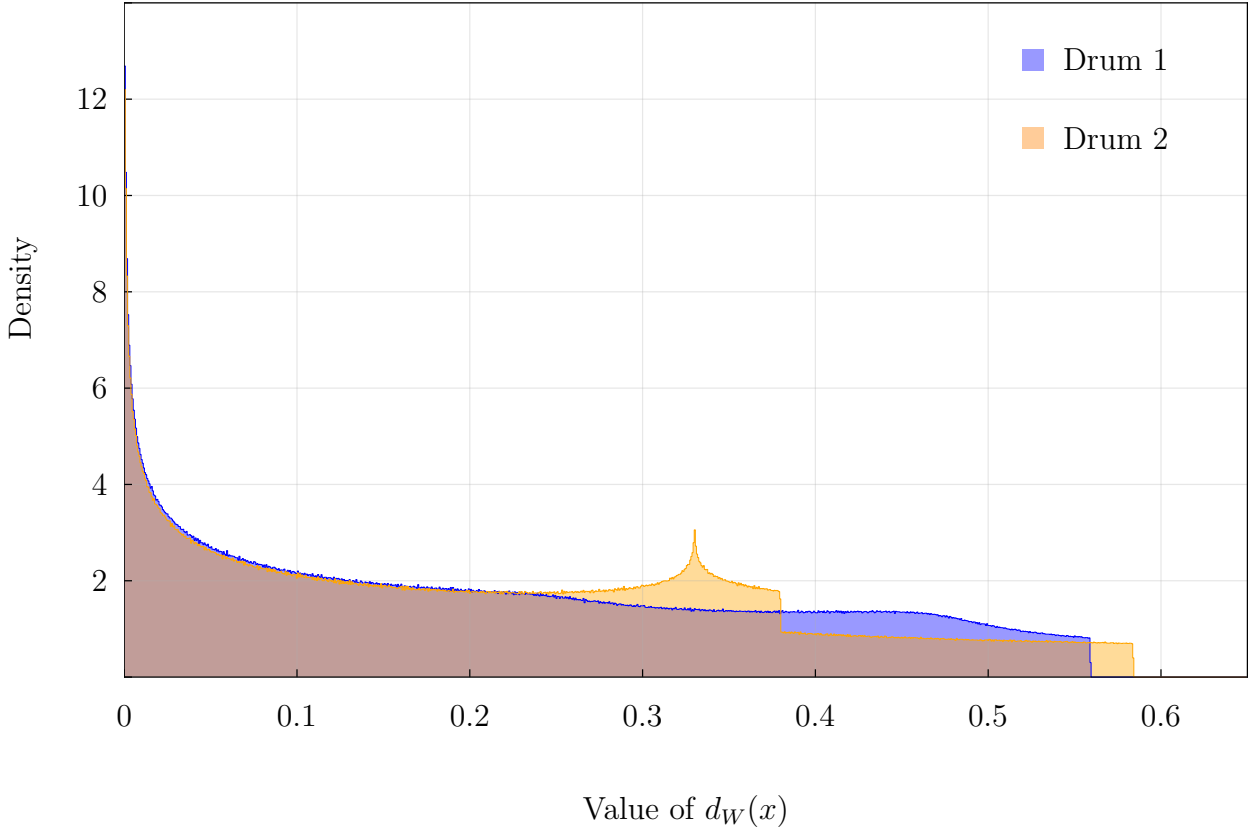
We show that this is indeed the case by using the short-time asymptotics of the heat kernel, which is given by Varadhan's formula. 

For these two isospectral drums, it can be shown that $R_1<R_2$, where $R_1$ and $R_2$ are the radii of the largest inscribed circles of $\Omega_1$ and $\Omega_2$, respectively. Let $x_i$ denote the center of the respective circle, and let $B_{R_i}\coloneqq B_{R_i}(x_i)\subset\Omega_i$ be the disk contained within the drum. Note that $d_{W_i}(x)\equiv w_i(x,t_0)$ as a function of $x\in\Omega_i$, where $w_i(x,t)=\int_{\Omega_i} H_i(x,y,t)\,\mathrm d\mu(y)$ is the solution to the heat equation
\begin{equation*}
    \begin{cases}
        u_t(x,t) = \Delta u(x,t), & (x,t) \in \Omega_i \times (0,\infty), \\
        u(x,t) = 0, & (x,t) \in \partial\Omega_i \times (0,\infty), \\
        u(x,0) = 1, & x \in \Omega_i^\circ.
    \end{cases}
\end{equation*}

Since $L=\Delta$ is also the generator for a Brownian motion $X_t$ with $\mathbb E X_t=0$ and $\text{Var} (X_t)=2t$, we find $1-w_i(x,t) = \mathbb P_x(\tau_{D}\leq t)$, where $\tau_{D}$ is the time $X_t$ exits $D\subset\Omega_i$ when started at $x\in D$. Due to~\cite{hsu2002} for any $r<R_i$ we find the following  
\begin{equation*}
    \mathbb P_x(\tau_{B_r(x)}\leq t) \sim c_r(x) \exp\Big(-\frac{r^2}{4t}\Big),
\end{equation*}
such that, for small enough $t$, we have
\begin{equation*}
    \mathbb P_{x_2}(\tau_{\Omega_2}\leq t)\leq \mathbb P_{x_2}(\tau_{B_{R_2}}\leq t)\leq C_2\exp\Big(-\frac{R_2^2}{4t}\Big).
\end{equation*}
Similarly, on $\Omega_1$ we have
\begin{equation*}
    \mathbb P_x(\tau_{\Omega_1}\leq t) \geq C_1 \exp\Big(-\frac{d(x,\partial \Omega_1)^2}{4t}\Big) \geq C_1 \exp\Big(-\frac{R_1^2}{4t}\Big)
\end{equation*}
for any $x\in\Omega_1$. Thus there exists $t>0$ small enough such that for all $x\in\Omega_1$ we have $w_1(x,t)< w_2(x_2,t)$, as $R_1<R_2$. Hence for the appropriate $t_0$ we obtain differing continuous spectra for $K_{W_1}$ and $K_{W_2}$.

\bibliography{bibliography}
\bibliographystyle{amsplain}

\end{document}